\documentclass[11pt]{article}

\usepackage{amsthm,amsmath}
\usepackage{natbib}
\usepackage{amsfonts}
\usepackage{amssymb}
\RequirePackage[OT1]{fontenc}
\usepackage[usenames,dvipsnames]{xcolor}
\usepackage{setspace}

\numberwithin{equation}{section}
\theoremstyle{plain}
\newtheorem{thm}{Theorem}[section]
\theoremstyle{remark}
\newtheorem{rem}{Remark}[section]

  \theoremstyle{remark}

  \theoremstyle{plain}
  \theoremstyle{definition}
  \newtheorem{example}[thm]{\protect\examplename}
  \newtheorem{lemma}[thm]{Lemma
  }
\usepackage[english]{babel}
  \providecommand{\examplename}{Example}
\def\ci{\perp\!\!\!\perp}
\usepackage[tmargin=1.0in,bmargin=1.0in,rmargin=1.0in,lmargin=1.0in]{geometry}

\title{Duality in Graphical Models}
\author{Dhafer Malouche\thanks{Ecole Sup\'{e}riore de la Statistique et de l'Analyse de l'Information,  Universit\'{e} de Carthage, Tunisia, \texttt{dhafer.malouche@me.com}} 
  \and Bala Rajaratnam \thanks{Department of Statistics, Stanford University, Stanford CA 94305, \texttt{brajarat@stanford.edu} }
  \and Benjamin T. Rolfs \thanks{Institute for Computational and Mathematical Engineering, Stanford University, Stanford CA 94305, \texttt{benrolfs@stanford.edu} }}

\begin{document}

\maketitle

\begin{abstract}
Graphical models have proven to be powerful tools for representing high-dimensional systems of random variables. One example of such a model is the undirected graph, in which lack of an edge represents conditional independence between two random variables given the rest. Another example is the bidirected graph, in which absence of edges encodes pairwise marginal independence. Both of these classes of graphical models have been extensively studied, and while they are considered to be dual to one another, except in a few instances this duality has not been thoroughly investigated. In this paper, we demonstrate how duality between undirected and bidirected models can be used to transport results for one class of graphical models to the dual model in a transparent manner. We proceed to apply this technique to extend previously existing results as well as to prove new ones, in three important domains. First, we discuss the pairwise and global Markov properties for undirected and bidirected models, using the pseudographoid and reverse-pseudographoid rules which are weaker conditions than the typically used intersection and composition rules. Second, we investigate these pseudographoid and reverse pseudographoid rules in the context of probability distributions, using the concept of duality in the process. Duality allows us to quickly relate them to the more familiar intersection and composition properties. Third and finally, we apply the dualization method to understand the implications of faithfulness, which in turn leads to a more general form of an existing result. 
\end{abstract}




\onehalfspacing

\section{Introduction}

Graphical models are used to study conditional independence statements about a set of $p < \infty$ random variables $\mathbf{X} = \{X_v\}_{v\in V}$ where $V = \{1,2,\dots,p\}$. Most generally, a graphical model for such a system is a graph $\mathcal{G} = (V,E)$. In such a graph, inclusion of $(a,b)$ in the edge set $E$ is predicated on some conditional independence statement about $X_a$ and $X_b$. For example, the undirected graph corresponding to $\mathbf{X}$ (called the concentration graph in some settings, such as \citep{Pearl1986}) is constructed such that $(a,b) \not\in E$ if and only if $X_a$ and $X_b$ are conditionally independent given the rest of the variables. The bidirected graph (called the covariance graph in some settings, such as \citep{Cox1993, Kauermann1996}), is constructed using the rule $(a,b) \not \in E$ if and only if $X_a$ and $X_b$ are independent. 

Undirected and bidirected graphical models have been widely studied in the literature \cite[see][for references]{Cox1996, Lauritzen1996, WhittakerBook}. Important to the study of such models are the equivalence between the pairwise Markov and global Markov properties (see Section \ref{subsec:markov_props}), general conditions for such equivalences, and the concept of faithfulness. Under certain conditions, graphical models defined using pairwise relations also encode more complicated global conditional independence statements. These conditional independence statements are typically represented by separation statements in an undirected or bidirected graph $\mathcal{G}$, and when this occurs the random variable $\mathbf{X}$ is said to be globally Markov with respect to $\mathcal{G}$. More specifically, when the global Markov condition is satisfied, separation of two disjoint subsets $A$ and $B$ given a third separating subset $S$ implies a conditional independence statement about $\mathbf{X}_A = \{X_a\}_{a \in A}$, $\mathbf{X}_B$, and $\mathbf{X}_S$. The ability of a graphical model $\mathcal{G}$ to encode such complex conditional independence statements is important when using such models in applications. The reverse containment, that all conditional independence statements within a random vector being encoded by separation statements in a given graph, is known as faithfulness.

Several authors have specified conditions under which the pairwise and global Markov properties are equivalent in both undirected and bidirected models \citep[][and others]{Kauermann1996, Pearl1986}. Conditions under which a distribution is faithful to a graphical model have also been formulated for undirected and bidirected trees \citep[and others]{Becker2005, Malouche2011}. However, although undirected and bidirected graphical models are known to be dual to each other (a notion formalized by \cite{Matus1992a}), they have frequently been treated differently -- especially when proving properties of such models. In several instances, authors have succeeded in obtaining results for bidirected graphs that parallel those for undirected graphs, but as we shall demonstrate have used more complicated proof techniques than necessary. In this paper, we demonstrate that many of these results could have been achieved using the dual framework of \citet{Matus1992b}, and use this formalism to develop even more general results. Our approach shows how without exception, results on bidirected graphs can be easily obtained by analyzing undirected graphs.

Besides introducing and investigating the important technique of using duality to ``transport" results in the undirected graph setting to the bidirected graph setting, we also enumerate below the additional novel contributions in the paper. First, weaker conditions -- the so-called pseudographoid rules -- for the equivalence between the pairwise and global Markov conditions on both undirected and bidirected graphs is given. Duality is used to adapt the result on undirected graphs to the dual result on bidirected graphs. Second, the relationship between the familiar intersection/composition properties and the  more general pseudographoid rules is formally derived. Third, a result on faithfulness in the bidirected graph setting, known to be true only for Gaussian random variables, is generalized in a significant way. In many cases, ``direct" proofs are given side by side with proofs using duality, to underscore the power of the technique.

The paper is organized as follows. In Section \ref{sec:prelims}, we introduce preliminary notation and concepts. The general approach taken by the paper is outlined in Section \ref{sec:duality}. The equivalence between pairwise and global Markov properties is considered in Section \ref{sec:pseudo_global}. These conditions, called the pseudographoid and reverse pseudographoid rules, are studied in greater detail in Section \ref{sec:semi_and_pseudo}. Finally, the idea of faithfulness is investigated in Section \ref{sec:faithful_trees}, and the use of duality allows extension of a result which formerly applied only to Gaussian distributions to a more general setting.

\section{Preliminaries}
\label{sec:prelims}
In this section, notation concerning conditional independence structures and graphical models is introduced. These are then stated in the language of relations as in \cite{Matus1992b}, a formalism which is used throughout this paper. Some preliminary results regarding various closure rules on the set of relations are derived. 

\subsection{Conditional Independence}
\label{subsec:cond_independence}

In the sequel, we will distance ourselves from conditional independence structures of probability distributions and instead work with general relations (defined in Section \ref{subsec:relations}). However, it is useful to use probability distributions as a motivating example before moving to the language of relations.

Throughout this paper, $\mathbf{X} = \{X_v\}_{v\in V}$ is a random vector indexed by a finite set $V$. For any $A\subseteq V$, $\mathbf{X}_A = \{X_a\}_{a \in A}$ denotes a sub-vector of $\mathbf{X}$. For disjoint $A, B, C \subseteq V = \{1,\dots,p\}$, with $\mathbf{X}_A$, $\mathbf{X}_B$, and $\mathbf{X}_C$ taking values in $\sigma$-algebras $\mathcal{S}_A$, $\mathcal{S}_B$, and $\mathcal{S}_C$ respectively, $X_A$ is said to be conditionally independent of $X_B$ given $X_C$ if $\forall S_A \in \mathcal{S}_A, S_B \in \mathcal{S}_B, S_C \in \mathcal{S}_C$, 
\begin{align*}
P(\mathbf{X}_A \in S_A, \mathbf{X}_B \in S_B | \mathbf{X}_C \in S_C) & = P(\mathbf{X}_A \in S_A | \mathbf{X}_C \in S_C) P(\mathbf{X}_B \in S_B | \mathbf{X}_C \in S_C),  
\end{align*}
where $P$ is the law of $\mathbf{X}$. In this case we say that
\begin{align*}
\mathbf{X}_A \ci \mathbf{X}_B | \mathbf{X}_C,
\end{align*}
which will henceforth be shortened to $A \ci B | C$ when there is no ambiguity. 

Any system of random variables satisfies the following \emph{semigraphoid rules} \citep{Pearl1986}:
\begin{itemize}
\item[1.] Symmetry $(S)$: $A \ci B | S \iff B \ci A | S$
\item[2.] Decomposition $(D)$: $ A \ci BC | S \Rightarrow A \ci B | S$
\item[3.] Weak Union $(U)$: $A \ci BC | S \Rightarrow A \ci B | SC$. 
\item[4.] Contraction $(C)$: $A \ci B | S \land A \ci C | BS \Rightarrow A \ci BC | S$ ,
\end{itemize}
where $BC$ is shorthand for $B\cup C$ when there is no ambiguity, and all sets $A,B,C,S$ above are disjoint. 
Above, $B$ and $C$ can represent any partition of $BC = B\cup C$ with $B\cap C = \emptyset$. Note that although any set of random variables necessarily satisfies the semigraphoid axioms above, the list is not complete; in fact, no finite characterization of random variables using such rules exists \citep{Studeny1992}. 

Any random vector $\mathbf{X}$ also satisfies the following \emph{localizability} property \cite[see][Lemma 3]{Matus1992b}: 
\begin{itemize}
\item[5.] Localizability $(L)$: $A \ci B | S \iff a \ci b | S^\prime, \ \forall a \in A, b \in B, S \subseteq S^\prime \subseteq SAB\backslash ab,$
\end{itemize}
where lowercase letters, above and henceforth, represent singletons, i.e., $a = \{a\}$. More generally, \cite{Matus1992b, Matus1997} demonstrates that any semigraphoid is localizable.

Furthermore, $\mathbf{X}$ may satisfy 
\begin{itemize}
\item[6.] Intersection (I): $A \ci B |SC \land A \ci C |SB \Rightarrow A \ci BC | S$
\item[7.] Composition (M): $A\ci B |S \land A \ci C |S \Rightarrow A \ci BC |S$.
\end{itemize}
If $\mathbf{X}$ admits a strictly positive density, it follows that $\mathbf{X}$ satisfies the intersection rule \citep{Lauritzen1996}. It is well-known that Gaussian random variables satisfy both intersection and composition. Some simple results relating the above properties are presented in Section \ref{subsec:relations}.

\subsection{Relations}
\label{subsec:relations}

In the sequel, it will be useful to abstract conditional independence structures using the notation of \emph{relations}. For a finite set $V$, define $\mathcal{T}(V)$ as the set of all triples $(A,B|C)$ with $A,B,C \subseteq V$ pairwise disjoint and $A,B$ both nonempty. A subset of $\mathcal{T}(V)$ will be referred to as a \emph{relation}. For a random vector $\mathbf{X}$ indexed by $V$, define the relation $[\mathbf{X}] \subseteq \mathcal{T}(V)$ by $(A,B|C) \in [\mathbf{X}] \iff A \ci B | C$ (where $A,B,C \subseteq V$ are disjoint) which is called the \emph{conditional independence structure} of $\mathbf{X}$. As all relations in general will be subsets of $\mathcal{T}(V)$, the dependence on $V$ will be henceforth suppressed.

All of the closure rules of Section \ref{subsec:cond_independence} can translated into the notation of relations. For example,  $\mathcal{L} \subseteq \mathcal{T}$ is a semigraphoid if it is closed under the semigraphoid rules in the sense that
\begin{align*}
&(S): \ (A, B | S) \in \mathcal{L} \iff (B, A | S) \in \mathcal{L}\\
&(D): \ (A, BC | S) \in \mathcal{L} \Rightarrow (A, B | S) \in \mathcal{L}\\
&(U): \ (A, BC | S) \in \mathcal{L} \Rightarrow (A, B | SC) \in \mathcal{L}\\
&(C): \ (A, B | S), (A, C | BS) \in \mathcal{L} \Rightarrow (A, BC | S) \in\mathcal{L}. 
\end{align*}
The following simple lemma (stated by \citet{Matus1992a} without proof) gives a compact representation of properties $(D), (U)$, and $(C)$ above.
\begin{lemma}
\label{lem:semi_parsim}
A relation $\mathcal{L}$ is closed under decomposition $(D)$, weak union $(U)$, and contraction $(C)$ if and only if it holds that 
\begin{align*}
(DUC): \ (A, BC | S) \in \mathcal{L} \iff (A, B | S), (A, C | SB) \in \mathcal{L}. 
\end{align*}
\end{lemma}
\begin{proof}
$\Rightarrow:$ Suppose first that $\mathcal{L}$ is closed under $(D), (U)$, and $(C)$. If $(A, BC | S) \in \mathcal{L}$, it follows by $(D)$ that $(A, B | S) \in \mathcal{L}$ and by $(U)$ that $(A, B | SC) \in \mathcal{L}$. On the other hand, if $(A, B |S), (A, C | SB) \in \mathcal{L}$, it follows by $(C)$ that $(A, BC | S) \in \mathcal {L}$. 

$\Leftarrow:$ Now suppose that $\mathcal{L}$ is closed under $(DUC)$ defined in the statement of the lemma. Then it is clear by definition that $\mathcal{L}$ is closed under $(D), (U)$, and $(C)$.
\end{proof}
This representation of the semigraphoid rules is parsimonious in the sense that it consists of one if and only if statement. We shall later see that this is convenient when proving results concerning semigraphoids. 

Furthermore, $\mathcal{L}$ is localizable if 
\begin{align*}
(L): \ (A,B|S) \in \mathcal{L} \iff (a,b|S^\prime) \in \mathcal{L} , \ \forall a \in A, b \in B, S \subseteq S^\prime \subseteq SAB\backslash ab. 
\end{align*}
Note that for any random variable $\mathbf{X}$, the corresponding conditional independence structure $[\mathbf{X}]$ is a semigraphoid \citep{Pearl1986}. Under the localizability condition, \emph{global} statements about general triples $(A,B|C)$ can be constructed from \emph{pairwise} triples of the form $(a,b|C)$, where $a,b$ are singletons. Specifically, define $\mathcal{S}(V):=\mathcal{S} = \{(a,b|C) \ : a \neq b, C \subseteq V \backslash ab\}$. Global statements about a localizable set of triples $\mathcal{L}  \subseteq \mathcal{T}$ can be made when working only with $\mathcal{S}(\mathcal{L}) :=\mathcal{L}  \cap \mathcal{S}$. This technique is used effectively, for example, in \cite{Matus1992b, Matus1997, Lnenicka2007}.

The following lemma shows that localizability is a strictly weaker condition than the semigraphoid rules. 
\begin{lemma}
\label{lem:semi_local_comparison}
Closure of a relation $\mathcal{L}$ under the semigraphoid axioms implies localizability. Closure of $\mathcal{L}$ under localizability implies decomposition $(D)$ and weak union $(U)$, but not contraction $(C)$. 
\end{lemma}
\begin{proof}
The fact that semigraphoids are localizable is shown in \cite{Matus1992b}, Lemma 3. See also \cite{Matus1997}, Lemma 1. 

Now let $\mathcal{L}$ be localizable. Then 
\begin{align*}
(A, BC |S) \in \mathcal{L} &\Rightarrow_{(L)} (a,b|S^\prime) \in \mathcal{L} , \ \forall a \in A, b \in BC, S \subseteq S^\prime \subseteq SABC\backslash ab \\
& \Rightarrow \ \ \ \ (A, B| S^\prime) \in \mathcal{L}, \ \forall a \in A, b \in B, S\subseteq S^\prime \subseteq SAB\backslash ab \\
&\Rightarrow_{(L)} (A, B | S) \in \mathcal{L}, 
\end{align*}
which shows closure of $\mathcal{L}$ under decomposition $(D)$. Moreover,
\begin{align*}
(A, BC |S) \in \mathcal{L}&\Rightarrow_{(L)} (a,b|S^\prime) \in \mathcal{L} , \ \forall a \in A, b \in BC, S \subseteq S^\prime \subseteq SABC\backslash ab \\ 
&\Rightarrow \ \ \ \ (a, b | S^\prime) \in \mathcal{L}, \ \forall a \in A, b \in B, SC\subseteq S^\prime \subseteq SABC \backslash ab \\
&\Rightarrow_{(L)} (A, B | SC) \in \mathcal{L}, 
\end{align*}
and so $\mathcal{L}$ is closed under weak union $(U)$.

It remains to show by counter example that $\mathcal{L}$ need not be closed under contraction. Consider the relation $\mathcal{L} = \{(a, b | \emptyset), (a, c | b) \}$ which is closed under localizability $(L)$. However, $\mathcal{L}$ is not closed under contraction, as it does not contain $(a, bc | \emptyset)$. 
\end{proof}
Localizability and its relation to the semigraphoid axioms is considered in more detail in Section \ref{sec:semi_and_pseudo}, when we consider localizability and the pseudographoid rules in detail. For more details on the rich mathematical framework underlying $\mathcal{T}$ and its subfamilies, we refer the reader to \cite{Lnenicka2007}, \cite{Vomlel2007}, \cite{Studeny1997, Studeny2001} and \cite{Matus1992a, Matus1997}.

\subsection{Graphical Models}
\label{subsec:graphical_models}

Let $\mathcal{G} = (V,E)$ be an undirected graph with vertex set $V$ and edge set $E \subseteq V \times V$ such that $(u,v) \in E \iff (v,u) \in E$. Two distinct vertices $v,u \in V$ are said to be \emph{adjacent} in $\mathcal{G}$ if $(v,u) \in E$ and in this case we write $u\sim_\mathcal{G} v$ or simply $u\sim v$ when it is unambiguous. The vertices $u,v$ are said to be \emph{connected} in $\mathcal{G}$ if there exists some $u = w_1,\dots,w_{n+1} = v \in V$, with $w_i$ distinct, such that $w_i \sim w_{i+1}, \forall i \in 1,\dots,n$. In this case, $(w_1, w_2, \dots, w_{n+1})$ is said to be a \emph{path} connecting $v = w_1$ and $u = w_{n+1}$ in $\mathcal{G}$. We define $\mathcal{P}(u,v)$ as the set of all paths connecting $u$ and $v$. 

Given an undirected graph $\mathcal{G} = (V,E)$ and disjoint $A,B,S \subseteq V$, we say that $S$ \emph{separates} $A$ and $B$ in $\mathcal{G}$ if $\forall a\in A, b \in B$, any path in $\mathcal{P}(a,b)$ contains at least one element of $S$. In this case, we write
\begin{align*}
A \perp_\mathcal{G} B | S,
\end{align*}
and define $[\mathcal{G}] \subseteq \mathcal{T}$ by $[\mathcal{G}] := \{(A,B|S) \ : \ A \perp_\mathcal{G} B | S\}$. For any undirected graph $\mathcal{G}$, $[\mathcal{G}]$ is a semigraphoid closed under intersection and composition \citep{Pearl1986}. It follows that $[\mathcal{G}]$ is localizable; for a direct proof of this statement, see the appendix. For a further discussion of $[\mathcal{G}]$ in a larger context, see \citet{Matus1997}, \citet{Lnenicka2007} and \citet{Matus1992a}. 

Given a relation $\mathcal{L}$, we construct the undirected graph corresponding to $\mathcal{L}$, written as $\mathcal{G}_{un}(\mathcal{L}) = (V, E_{un}(\mathcal{L}))$, according to the rule
\begin{align*}
(a,b) \not\in E_{un}(\mathcal{L}) \iff (a,b|V\backslash ab) \in \mathcal{L}. 
\end{align*}
The \emph{bidirected graph} corresponding to $\mathcal{L}$ is the graph $\mathcal{G}_{bi}(\mathcal{L}) = (V,E_{bi}(\mathcal{L}))$ constructed according to the rule
\begin{align*}
(a,b) \not\in E_{bi}(\mathcal{L}) \iff (a,b|\emptyset) \in \mathcal{L}. 
\end{align*}
Note that each edge of $\mathcal{G}_{bi}(\mathcal{L})$ can be thought of as a bidirected edge. We use this convention to be consistent with previous work on directed graphs, and adhere to the notation of \citet{Lauritzen1996}. Some authors refer to the undirected graph above as the \emph{concentration graph} and the bidirected graph as the \emph{covariance graph} \citep{Cox1996}. This terminology makes sense when modelling Gaussian distributions, for which pairwise partial independence between random variables is encoded by sparsity in the inverse covariance (concentration) matrix. Since the context of this paper is much broader, we use the undirected/bidirected nomenclature. 

\begin{example}
Typically, the relation $\mathcal{L}$ being considered is the conditional independence structure for some random variable $\mathbf{X}$. The undirected and bidirected graphs corresponding to $\mathbf{X}$ are constructed using pairwise partial or marginal independences present in $[\mathbf{X}]$. Specifically, the undirected graph $\mathcal{G}_{un}([\mathbf{X}]) = (V,E_{un}([\mathbf{X}]))$ satisfies $(a,b) \not\in E_{un}([\mathbf{X}]) \iff X_a \ci X_b | \mathbf{X}_ {V\backslash ab}$. Alternatively, $(a,b) \in E_{un}([\mathbf{X}]) \iff X_a \not\ci X_b |\mathbf{X}_ {V\backslash ab}$.  On the other hand, the bidirected graph $\mathcal{G}_{bi}([\mathbf{X}]) = (V,E_{bi}([\mathbf{X}]))$ satisfies $(a,b) \in E_{bi}([\mathbf{X}]) \iff X_a \ci X_b$.
\end{example}

Such graphs have been widely studied going back to \cite{Cox1993}, \cite{Kauermann1996}, and \cite{Pearl1986}. For general reference on both undirected and bidirected graphs, as well as other types of graphical models, see for example \cite{Cox1996}, \cite{Lauritzen1996}, \cite{Studeny2001}, and \cite{WhittakerBook}.

\subsection{Global Markov Properties}
\label{subsec:markov_props}

Under certain assumptions, pairwise statements about graphical models can be used to construct global statements. Given a relation $\mathcal{L}$ and graph $\mathcal{G} = (V,E)$, we say that $\mathcal{L}$ is:
\begin{itemize}
\item[a.] \emph{Undirected (concentration) pairwise Markov} with respect to $\mathcal{G}$ if $(a,b) \notin E \Rightarrow (a,b|V\backslash ab) \in \mathcal{L}$, that is, $\mathcal{S}(\mathcal{L}) \subseteq [\mathcal{G}]$.  
\item[b.] \emph{Undirected global Markov} with respect to $\mathcal{G}$ if $A \perp_\mathcal{G} B | S \Rightarrow (A,B|S) \in \mathcal{L}$, that is, $[\mathcal{G}] \subseteq \mathcal{L}$.  
\item[c.] \emph{Bidirected (covariance) pairwise Markov} with respect to $\mathcal{G}$ if $(a,b)\notin E \Rightarrow (a,b|\emptyset) \in \mathcal{L}$. 
\item[d.] \emph{Bidirected global Markov} with respect to $\mathcal{G}$ if $A \perp_{\mathcal{G}} B | S \Rightarrow (A,B| V\backslash ABS)  = (A,B|S)^\rceil \in \mathcal{L}$. 
\end{itemize}

Note that we speak of Markov properties with respect to a relation $\mathcal{L}$ as compared to a random variable $\mathbf{X}$. In the language of conditional independence, the undirected global Markov rule, e.g., is the statement that
\begin{align*}
A \perp_\mathcal{G} B | S \Rightarrow \mathbf{X}_A \ci \mathbf{X}_B | \mathbf{X}_S. 
\end{align*}

Fixing some arbitrary $\mathcal{G} = (V,E)$, note that if $\mathcal{L}$ is undirected  pairwise Markov with respect $\mathcal{G}$ then $E_{un}(\mathcal{L}) \subseteq E $; $\mathcal{G}_{un}(\mathcal{L})$ is minimal in this sense. Further, if $\mathcal{G}^\prime = (V, E^\prime)$ satisfies $E_{un}([\mathbf{X}]) \subseteq E^\prime$, then $\mathcal{L}$ is undirected pairwise Markov with respect to $\mathcal{G}^\prime$. Similarly, if $\mathcal{L}$ is bidirected pairwise Markov with respect to $\mathcal{G}$, then $E_{bi}(\mathcal{L}) \subseteq E$. If $E_{bi}([\mathbf{X}]) \subseteq E^\prime$, then $\mathcal{L}$ is bidirected pairwise Markov with respect to $\mathcal{G}^\prime$.

It is well-known \citep{Pearl1986} that closure of a relation $\mathcal{L}$ under intersection is a sufficient condition for equivalence between the undirected pairwise and undirected global Markov properties. The same is true \citep{Kauermann1996,Banerjee2003} for relations closed under composition with respect to the bidirected Markov properties. In Section \ref{sec:pseudo_global}, the assumptions of intersection and composition will be weakened. In addition, the concept of duality will be used to demonstrate how the result on bidirected graphs follows directly from that on undirected graphs.

\subsection{Additional Closure Rules}
\label{subsec:pseudographoids}

We now present four additional closure rules on $\mathcal{L} \subseteq \mathcal{T}$. First, we define the pseudographoid and reverse pseudographoid rules. Pseudographoids have been previously studied going back at least to \citet{Matus1997}, and are defined here in their pairwise form as in \citet{Lnenicka2007}. The reverse pseudographoid rule is also considered by \citet{Lnenicka2007}. 

\begin{itemize}
\item[8.] Pseudographoid Rule $(P)$: $(a,b|Sc), (a,c|Sb) \in \mathcal{L} \Rightarrow (a,b|S), (a,c|S) \in \mathcal{L}$, where $a,b,c$ are singletons. 
\item[9.] Reverse Pseudographoid Rule $(R)$: $(a,b|S), (a,c|S) \in \mathcal{L} \Rightarrow (a,b|Sc), (a,c|Sb) \in \mathcal{L}$ where $a,b,c$ are singletons. 
\end{itemize}

Note that the pseudographoid rule is  analogous to intersection, but with some sets restricted to be singletons. On the other hand, the reverse pseudographoid rule is a weakened version of composition. In Section \ref{sec:pseudo_global}, these rules are shown to be both necessary sufficient for equivalence of the Markov properties (Section \ref{subsec:graphical_models}). In Section \ref{sec:semi_and_pseudo}, the pseudographoid and reverse pseudographoid rules are studied with respect to semigraphoids, and shown to be equivalent to intersection (respectively, composition) in that context. 

In the sequel, the following two closure properties will be used to study faithfulness (see Section \ref{sec:faithful_trees}) of graphical models:
\begin{itemize}
\item[10.] Decomposable Transitivity: $(aB,De|c), \ (a,e|BD) \in \mathcal{L} \Rightarrow (a,c|B) \in \mathcal{L} \lor (c,e|D) \in \mathcal{L}$. 
\item[11.] Dual Decomposable Transitivity: $(aB,De|V\backslash aBcDe), \ (a,e|V\backslash aBDe) \in \mathcal{L} \Rightarrow (a,c|V\backslash aBc)\in \mathcal{L} \lor (c,e|V\backslash cDe) \in \mathcal{L}$. 
\end{itemize}
Decomposable transitivity was first defined in \citet{Becker2005} to provide conditions for faithfulness of random variables to tree graphs. Dual decomposability is a new notion introduced in this paper to facilitate derivation of a dual result to that of \citet{Becker2005} in Section \ref{sec:faithful_trees}.

\section{Duality}
\label{sec:duality}

We now introduce the notion of duality, a tool which we rely heavily on in the sequel. Given some triple $(A,B|C) \in \mathcal{T}$, define its \emph{dual} by $(A,B|C)^\rceil :=(A,B|V \backslash ABC)$. Similarly, given a relation $\mathcal{L} \subseteq \mathcal{T}$, its dual is defined to be 
\begin{align*}
{L}^\rceil  := \{(A,B|C) \ : \ (A,B|V\backslash ABC)\in\mathcal{L} \}.
\end{align*}
Note that $(\mathcal{L}^\rceil)^{\rceil} = \mathcal{L}$. Duality is discussed in detail in \cite{Matus1992a, Lnenicka2007}, in the context of various classes of relations.

Consider the undirected and bidirected graphs corresponding to a given relation $\mathcal{L}$, as defined in Section \ref{subsec:graphical_models}. Note that $(a,b|V\backslash\{a,b\})^\rceil = (a,b|\emptyset)$, and hence the undirected graph  corresponding to $\mathcal{L}$ is the dual of the bidirected graph (and vice versa), that is $[\mathcal{G}_{un}]^\rceil = [\mathcal{G}_{bi}]$. Furthermore, the Markov properties for bidirected graphs are simply defined using duality. A relation $\mathcal{L}$ is bidirected pairwise Markov with respect to some $\mathcal{G} = (V, E)$ if $\mathcal{S}(\mathcal{L})^\rceil \subseteq [\mathcal{G}]$, and bidirected global Markov with respect to $\mathcal{G}$ if $[\mathcal{G}]^\rceil \subseteq \mathcal{L}$. 

We now define a sense in which one can ``dualize" a result regarding undirected graphs to one about bidirected graphs. Let $\mathcal{L} \subseteq \mathcal{T}$ be a relation. As in \cite{Studeny1997}, we define a \emph{rule with $r$ antecedents}, $\mathcal{R}$, as a set of $(r+1)$-tuples of $V$  with each $t_i$ a distinct triple on $V$. We say that a relation $\mathcal{L}$ is closed under $\mathcal{R}$ if for every $(r+1)$-tuple $\{t_1, \dots, t_{r+1}\} \in \mathcal{R}$, it holds that
\begin{align*}
t_i \in \mathcal{L},\ i = 1,\dots r \Rightarrow t_{r+1} \in \mathcal{L}. 
\end{align*} 
For example, the intersection rule is the set of all $3$-tuples of the form 
\[
\{(A,B|CD), (A,C|BD), (A,BC|D)\},
\] 
where $A,B,C,D$ are arbitrary disjoint subsets of $V$. We will also allow for $r=0$, and in this case $\mathcal{L}$ is closed under the unary rule $\mathcal{R}$ when 
\begin{align*}
\{t_1\} \in \mathcal{R} \Rightarrow t_1 \in \mathcal{L}. 
\end{align*}
which is simply set-containment. 

The dual of a rule $\mathcal{R}$ is defined as
\begin{align*}
\{ \{t_1^\rceil, \dots, t_{r+1}^\rceil \} \ : \ \{t_1, \dots, t_{r+1} \} \in \mathcal{R} \}.
\end{align*}
Noting that
\begin{align*}
\{(A,B|CD), (A,C|BD), (A,BC|D)\}^\rceil &= \{(A,B|V\backslash ABCD), (A,C|V\backslash ABCD), (A,BC|V\backslash ABCD)\},
\end{align*}
it is easy to see that the dual of the intersection rule is composition, i.e., the set of all triples of the form
\[
\{(A,B|D), (A,C|D), (A,BC|D) \}
\]
after the change of variables $V\backslash ABCD \rightarrow D$. 

Note that the closure of a relation $\mathcal{L}$ under some rule $\mathcal{R}$ is equivalent to closure of $\mathcal{L}^\rceil$ under $\mathcal{R}^\rceil$. Similarly, suppose it holds that whenever a relation $\mathcal{L}$ is closed under some finite set of rules, $\{\mathcal{R}_k\}_{k=1}^K$, it must also be closed under $\mathcal{R}_{K+1}$. Then it immediately holds that whenever a relation $\mathcal{L}$ is closed under $\{\mathcal{R}^\rceil_k\}_{k=1}^K$, it must also be closed under $\mathcal{R}^\rceil_{K+1}$. Otherwise, there would exist some $\mathcal{L}$ which was closed under $\{\mathcal{R}^\rceil_k\}_{k=1}^K$ but not under $\mathcal{R}^\rceil_{K+1}$. In this case, $\mathcal{L}^\rceil$ would be closed under $\{\mathcal{R}_k\}_{k=1}^K$ but not under $\mathcal{R}_{K+1}$, contradicting the assumption.

For example, for a given graph $\mathcal{G}$, consider the unary rules
\begin{align*}
\mathcal{R}_{\mathcal{G}}^{un,pair} = \{(a,b|V\backslash ab) \ : \ a \perp_\mathcal{G} b | V\backslash ab\}
\end{align*}
and
\begin{align*}
(\mathcal{R}_{\mathcal{G}}^{un,pair})^\rceil &= \{(a,b|\emptyset) \ : \ a \perp_\mathcal{G} b | S\} \\
&:= \mathcal{R}_{\mathcal{G}}^{bi,pair}.
\end{align*}
These are the pairwise undirected and pairwise bidirected Markov properties of Section \ref{subsec:markov_props}; they are jointly dual. Similarly, the global Markov properties are given by
\begin{align*}
\mathcal{R}_{\mathcal{G}}^{un,global} = \{(A,B|S) \ : \ A \perp_\mathcal{G} B | S\}
\end{align*}
and
\begin{align*}
(\mathcal{R}_{\mathcal{G}}^{un,global})^\rceil &= \{(A,B|V\backslash ABS) \ : \ A \perp_\mathcal{G} B | S\} \\
&:= \mathcal{R}_{\mathcal{G}}^{bi,global}.
\end{align*}
Hence, the existence of a result providing equivalence between the undirected pairwise and undirected global Markov rules under the assumption of some closure rule yields the dual result, i.e., equivalence between the bidirected pairwise and bidirected global Markov rules under the dual closure rule. Although we will use the technique described above in the sequel, we will not require the preceding notation regarding rules.

Before proceeding, we provide three lemmas which will be used in the remainder of this paper. The following lemma specifies the effect of the dualization operator on relations with respect to various closure rules.
\begin{lemma}
\label{lem:dual_operator}
A relation $\mathcal{L} \subseteq \mathcal{T}(V)$ is:
\begin{enumerate}
\item localizable if and only if $\mathcal{L}^\rceil$ is localizable. 
\item a semigraphoid if and only if $\mathcal{L}^\rceil$ is a semigraphoid.
\item closed under intersection if and only if $\mathcal{L}^\rceil$ is closed under composition. 
\item closed under the pseudographoid rule if and only if $\mathcal{L}^\rceil$ is closed under the reverse pseudographoid rule. 
\item closed under decomposable transitivity if and only if $\mathcal{L}^\rceil$ is closed under dual decomposable transitivity.
\end{enumerate}
\end{lemma}
\begin{proof}
$(1):$ It is clear by definition that $(\mathcal{L}^\rceil)^\rceil = \mathcal{L}$, therefore it suffices to show one direction. Therefore, assume that $\mathcal{L}$ is localizable; we will show that $\mathcal{L}^\rceil$ is as well.

We first show the $(\Rightarrow)$ direction of the localizability condition. Assume that $(A,B|S) \in \mathcal{L}^\rceil$, and consider any singletons $a\subseteq A, b\subseteq B$, and $S^\prime$ satisfying $S \subseteq S^\prime \subseteq SAB \backslash ab$.  We are done if $(a,b|S^\prime)  \in \mathcal{L}^\rceil$, which is equivalent to showing that $(a,b|V\backslash S^\prime ab) \in \mathcal{L}$. Since $(A,B|S) \in \mathcal{L}^\rceil \Rightarrow (A,B|V\backslash SAB) \in \mathcal{L}$, we have that for any $V\backslash SAB \subseteq \tilde{S} \subseteq V\backslash Sab$, $(a,b|\tilde{S}) \in \mathcal{L}$ by localizability of $\mathcal{L}$. But $S \subseteq S^\prime \subseteq SAB\backslash ab \Rightarrow V\backslash SAB \subseteq V\backslash S^\prime ab \subseteq V\backslash Sab$, and hence $(a,b|V\backslash S^\prime ab)$ as required. 

To show the $(\Leftarrow)$ direction, consider pairwise disjoint $A,B,S\subseteq V$, and assume that $(a,b|S^\prime) \in \mathcal{L}^\rceil$ for any singletons $a\subseteq A, b\subseteq B$ and $S\subseteq S^\prime \subseteq SAB\backslash ab$. Therefore, for any such $a,b,S^\prime$, it holds that $(a,b|V\backslash S^\prime ab) \in \mathcal{L}$. Hence, for any $V\backslash SAB \subseteq \tilde{S} \subseteq V\backslash Sab$, which implies by localizability of $\mathcal{L}$ that $(A,B| V\backslash SAB) \in \mathcal{L}$. Therefore, $(A,B|V\backslash SAB)^\rceil = (A,B|S) \in \mathcal{L}^\rceil$, completing the proof of $(1)$. 

$(2):$ Let $\mathcal{L}$ be a semigraphoid and consider $(A,BC|S) \in \mathcal{L}^\rceil$. Then $(A,BC|V\backslash SABC) \in \mathcal{L}$ which implies that $(A,C|V\backslash SABC), (A,B|V\backslash SAB) \in \mathcal{L}$ by decomposition and weak union, respectively. However, this in turn yields $(A,B|S), (A,C|SB) \in \mathcal{L}^\rceil$. Next, let $(A,B|S), (A,C|SB) \in \mathcal{L}^\rceil$ such that $(A,B|V\backslash SAB), (A,C|V\backslash SABC) \in \mathcal{L}$. By contraction, it follows that $(A,BC|V\backslash SABC) \in \mathcal{L}$, which implies that $(A,BC|S) \in \mathcal{L}^\rceil$, and hence $\mathcal{L}^\rceil$ is a semigraphoid. The reverse direction follows from the fact that $(\mathcal{L}^\rceil)^\rceil = \mathcal{L}$, completing the proof of $(2)$. 

$(3):$ Let $\mathcal{L}$ be a closed under intersection, and consider $(A,B|S), (A,C|S) \in \mathcal{L}^\rceil$. Then 
\[
(A,B|V\backslash SAB), (A, C|V\backslash SAC) \in \mathcal{L},
\]
and as $\mathcal{L}$ is closed under intersection this yields that $(A,BC|V\backslash SABC) \in \mathcal{L}$. It follows that $(A,BC|S) \in \mathcal{L}^\rceil$, and hence $\mathcal{L}^\rceil$ is closed under composition. 

Next, assume that $\mathcal{L}^\rceil$ is closed under composition, and let $(A,B|SC), (A,C|SB) \in \mathcal{L}$. Then
\[
(A,B|V\backslash SABC), (A,C|V\backslash SABC) \in \mathcal{L}^\rceil. 
\]
This implies by composition that $(A,BC|V\backslash SABC) \in \mathcal{L}^\rceil$, which in turn yields $(A,BC|S)\in \mathcal{L}$, completing the proof of $(3)$. 

$(4):$ Assume first that $\mathcal{L}$ is closed under $(P)$, and for pairwise disjoint $a,b,c,S \subseteq V$ with $a,b,c$ singletons let $(a,b|S), (a,c|S) \in \mathcal{L}^\rceil$. It follows that 
\begin{align*}
(a,b|V \backslash Sab), (a,c|V\backslash Sac) \in \mathcal{L}.
\end{align*}
Then by $(P)$, it also holds that 
\begin{align*}
(a,b| V\backslash Sabc), (a,c| V\backslash Sabc) \in \mathcal{L}.
\end{align*}
This in turn implies that $(a,b|Sc), (a,c|Sb) \in \mathcal{L}^\rceil$ which shows that $\mathcal{L}^\rceil$ is closed under $(R)$. 

Next assume that $\mathcal{L}^\rceil$ is closed under $(R)$, and $a,b,c,S$ as before let $(a,b|Sc), (a,c|Sb) \in \mathcal{L}$. Then 
\begin{align*}
(a,b|V\backslash Sabc), (a,c|V\backslash Sabc) \in \mathcal{L}^\rceil,
\end{align*}
and by $(R)$ it follows that
\begin{align*}
(a,b| V\backslash Sab), (a,c|V\backslash Sac) \in \mathcal{L}^\rceil.
\end{align*}
This implies that $(a,b|S), (a,c|S) \in \mathcal{L}$, and hence $\mathcal{L}$ is closed under $(P)$, completing the proof of $(4)$. 

$(5):$ This follows directly from the definition of dual decomposable transitivity. 
\end{proof}
Lemma \ref{lem:dual_operator} will greatly expedite proofs regarding undirected and bidirected graphs in the sequel. The following lemma from \citet{Lnenicka2007} allows dualization of any Gaussian random vector.
\begin{lemma}[Lemma 1 of \citet{Lnenicka2007}]
\label{lem:gaussian_duality}
Let $\mathbf{X}$ be a Gaussian random vector distributed as $\mathcal{N}(0,\Sigma)$. If $\mathbf{Y}$ is distributed as $\mathcal{N}(0,\Sigma^{-1})$, then $[\mathbf{Y}] = [\mathbf{X}]^\rceil$.
\end{lemma}
In particular, application of Lemma \ref{lem:gaussian_duality} yields the following result.
\begin{lemma}
\label{lem:gaussian_dual_closure}
Assume that for any Gaussian random vector $\mathbf{X}$, it holds that $[\mathbf{X}]$ is closed under some rule $\mathcal{R}$. Then for any Gaussian random vector $\mathbf{X}$, it also holds that $[\mathbf{X}]$ is closed under $\mathcal{R}^\rceil$. 
\end{lemma}
\begin{proof}
Assume that there exists some Gaussian random vector $\mathbf{X}$ for which $[\mathbf{X}]$ is not closed under $\mathcal{R}^\rceil$. Consider then the Gaussian random vector $\mathbf{Y}$ with $[\mathbf{Y}] = [\mathbf{X}]^\rceil$, the existence of which is guaranteed by Lemma \ref{lem:gaussian_duality}. Since closure of a relation $\mathcal{L}$ under $\mathcal{R}$ is equivalent to closure of $\mathcal{L}^\rceil$ under $\mathcal{R}^\rceil$, it must be the case that $[\mathbf{X}]^\rceil = [\mathbf{Y}]$ is not closed under $(\mathcal{R}^\rceil)^\rceil = \mathcal{R}$. Then $\mathbf{Y}$ is a Gaussian random vector for which $[\mathbf{Y}]$ is not closed under $\mathcal{R}$, contradicting our assumption. 
\end{proof}
As we will demonstrate in Section \ref{sec:faithful_trees}, Lemmas \ref{lem:gaussian_duality} and \ref{lem:gaussian_dual_closure} allow dualization of closure rules specific to the Gaussian distribution in general.

\section{Duality, Pseudographoid Rules and the Global Markov Properties}
\label{sec:pseudo_global}

In this section, the pairwise and global Markov properties are examined using the dualization technique developed in Section \ref{sec:duality}. We provide weaker conditions for the equivalences between pairwise and global Markov properties. The typical assumptions for this equivalence are the semigraphoid rules in combination with intersection for undirected graphs or composition for bidirected graphs. Instead, we use localizability (shown to be weaker than the semigraphoid rules by Lemma \ref{lem:semi_local_comparison}) in combination with either the pseudographoid or reverse pseudographoid rule (which are weaker than intersection/composition).

We begin by considering this equivalence for undirected graphs, and proceed to dualize the result to bidirected graphs. The equivalence between the pairwise and global Markov properties for bidirected graphs is treated differently in the literature than that for undirected graphs, but as we will see, the results are equivalent in the dual sense described in Section \ref{sec:duality}. The full proof is also given for bidirected graphs for completeness, although dualization is a far more efficient method. We first restate the following well known results regarding the global Markov properties in the language of relations. 

\begin{thm}[\cite{Pearl1986}] 
\label{thm:global_conc}
Let $\mathcal{L} \subseteq \mathcal{T}(V)$ be a semigraphoid closed under intersection $(I)$. Then $\mathcal{L}$ is undirected pairwise Markov with respect to a graph $\mathcal{G} = (V,E)$ if and only if $\mathcal{L}$ is also undirected global Markov with respect to $\mathcal{G}$. 
\end{thm}

\begin{thm}[\cite{Kauermann1996}, \cite{Banerjee2003}]
\label{thm:global_cov}
Let $\mathcal{L} \subseteq \mathcal{T}(V)$ be a semigraphoid closed under composition $(M)$. Then $\mathcal{L}$ is bidirected pairwise Markov with respect to a graph $\mathcal{G} = (V,E)$ if and only if $\mathcal{L}$ is also bidirected global Markov with respect to $\mathcal{G}$. 
\end{thm}

The assumptions of both theorems above can be separately weakened. Localizability can be used in place of the semigraphoid rules, while the pseudographoid and reverse pseudographoid rules can be used in place of intersection and composition, respectively. and each can be stated in terms of necessary and sufficient conditions. We consider both theorems above in light of these weakened properties. In doing so, we modify the logic to assume the pairwise Markov property, and then provide an equivalence between properties of $\mathcal{L}$ and the global Markov property. While this deviates from the literature, we find it a more natural framework from a graphical modelling perspective; when a graph is used to model a relation $\mathcal{L}$, the graph $\mathcal{G}$ is chosen such that $\mathcal{L}$ is pairwise Markov with respect to $\mathcal{G}$ and hence that should be assumed. 

\subsection{Pseudographoids and the Undirected Markov Properties}
\label{subsec:pseudo_global}

To begin, we adapt Theorem \ref{thm:global_conc} above, considering the undirected pairwise and undirected global Markov properties for any localizable relation closed under the pseudographoid rule. This equivalence was originally considered by \citep{Pearl1986} under the assumptions of the semigraphoid and intersection rules, and a related result is stated in the language of relations by \citet[Lemma 3]{Lnenicka2007}. The proof in one direction is technical, but essentially mirrors that of Theorem \ref{thm:global_conc} due to \citet{Pearl1986}; it has therefore been moved to the appendix. While \citet{Pearl1986} shows that closure under intersection is sufficient for the result, we instead assume closure under the pseudographoid rule. 

\begin{thm}
\label{thm:pseudo_conc}
Let $\mathcal{L} \subseteq{T}(V)$ be undirected pairwise Markov with respect to some graph $\mathcal{G} = (V,E)$. Then $\mathcal{L}$ is undirected global Markov with respect to $\mathcal{G}$ if and only if $\mathcal{L} \cap [\mathcal{G}]$ is localizable $(L)$ and closed under the pseudographoid rule $(P)$.
\end{thm}
\begin{proof}
$(\Leftarrow):$ See appendix.

$(\Rightarrow):$ Assume now that $\mathcal{L}$ is undirected global Markov with respect to $\mathcal{G}$. Then by definition, $\mathcal{L} \cap [\mathcal{G}] = [\mathcal{G}]$. By \citet[p. 108]{Matus1997}, $[\mathcal{G}]$ is a localizable pseudographoid, completing the proof. 
\end{proof}

\begin{rem}
\label{remark:conc_necessary}
Note that the necessary condition for equivalence between the undirected pairwise and undirected global Markov properties requires a condition on $\mathcal{L} \cap [\mathcal{G}]$, as opposed to all of $\mathcal{L}$. This is demonstrated by the next example. 
\end{rem}

\begin{example}
Let $\mathcal{G} = (V,E)$ be the path on four vertices, i.e., $V = \{a,b,c,d\}$ and $E = \{(a,b),(b,c),(c,d)\}$. In this case,
\begin{align*}
[\mathcal{G}] = \{(a,c|b), (a,d|b), (a,d|c), (a,d|bc), (b,d|c), (ab,d|c), (a,cd|b)\}.
\end{align*}
Consider then
\begin{align*}
\mathcal{L} = [\mathcal{G}] \cup \{(a,c|d)\}.
\end{align*}
Then $\mathcal{L}$ is both undirected pairwise and undirected global Markov with respect to $\mathcal{G}$, but is not closed under the pseudographoid rule, since it contains $(a,c|d)$ and $(a,d|c)$ but not $(a,c|\emptyset)$ and $(a,d|\emptyset)$. 
\end{example}

\subsection{Reverse Pseudographoids and the Bidirected Markov Properties}
\label{subsec:reverse_global}

The reverse pseudographoid rule is now examined in place of composition to relate the the bidirected pairwise and bidirected global Markov properties (see Theorem \ref{thm:pseudo_cov}). The original proof of such equivalence by \citet{Kauermann1996} showed sufficiency of the semigraphoid and composition rules, an assumption which is weakened here. We first provide a direct proof of this more general result. We then proceed to use the concept of duality to yield a much simpler proof of this general result. The direct proof in the $(\Leftarrow)$ direction of Theorem \ref{thm:pseudo_cov} below uses techniques similar in some ways to the proof of Theorem \ref{thm:global_cov} by \citet{Kauermann1996}. However, there are subtle differences. In fact it rather parallels the proof of Theorem \ref{thm:pseudo_conc} exactly in a dual sense. For this reason, and also to provide contrast with the brevity of the ensuing alternate proof which leverages duality, the direct proof of Theorem \ref{thm:pseudo_cov} is nevertheless given below. 

\begin{thm}
\label{thm:pseudo_cov}
Let $\mathcal{L} \subseteq{T}(V)$ be bidirected pairwise Markov with respect to some graph $\mathcal{G} = (V,E)$. Then $\mathcal{L}$ is bidirected global Markov with respect to $\mathcal{G}$ if and only if $\mathcal{L} \cap [\mathcal{G}]^\rceil$ is localizable $(L)$ and closed under the reverse pseudographoid rule $(R)$. 
\end{thm}
\begin{proof}
$(\Leftarrow):$ Assume that  $\mathcal{L} \cap [\mathcal{G}]^\rceil$ is $(L)$ and $(R)$. To begin, we show that $a \perp_\mathcal{G} b | S \Rightarrow (a,b|S)^\rceil \in \mathcal{L}$, where $a,b$ are singletons. As in the proof of Theorem \ref{thm:pseudo_conc}, this is done by induction on $|S|$. To begin, if $|S| = |V| - 2$ (i.e., $S = V\backslash\{ab\}$), it follows that $(a,b) \notin E$, which implies $(a,b|V\backslash ab)^\rceil = (a,b|\emptyset) \in \mathcal{L}^\rceil \cap [\mathcal{G}]$ by the bidirected pairwise Markov assumption.

Assume now that for singletons $a,b \subseteq V$ and $a,b\not\subseteq S^\prime$ with $|S^\prime| = k < |V| - 2$, $a \perp_\mathcal{G}  b | S^\prime \Rightarrow (a,b|S^\prime)^\rceil \in \mathcal{L}^\rceil \cap [\mathcal{G}]$. Then let $a\perp_\mathcal{G} b | S$ with $|S| = k - 1$. As $|S| < |V| - 2$, we can find some singleton $c \not \subseteq Sab$, and hence $a \perp_\mathcal{G} b | Sc$ and $a \perp_\mathcal{G}$. By the inductive hypothesis, this implies that $(a,b|Sc)^\rceil = (a,b|V \backslash Sabc) \in \mathcal{L}^\rceil \cap [\mathcal{G}]$. 

Next, as in the proof of Theorem \ref{thm:pseudo_conc}, $a \perp_\mathcal{G} b | S$ and $a \perp_\mathcal{G} b | Sc$ implies that either $a \perp_\mathcal{G} c | S$ or $c \perp_\mathcal{G} b | S$. Without loss of generality, let $a \perp_\mathcal{G} c | S$. Then $a \perp_\mathcal{G} c | Sb$, and by the inductive hypothesis, $(a,c|Sb)^\rceil = (a,c | V \backslash Sabc) \in \mathcal{L}^\rceil \cap [\mathcal{G}]$. Since $\mathcal{L}^\rceil \cap [\mathcal{G}]$ satisfies the reverse pseudographoid rule, $(a,b | V \backslash Sabc), (a,c | V \backslash Sabc) \in \mathcal{L}^\rceil \cap [\mathcal{G}]$ implies $(a,b| V\backslash Sab) = (a,b|S)^\rceil \in \mathcal{L}^\rceil \cap [\mathcal{G}]$. This completes the induction on $|S|$. 

Finally, assume that $A  \perp_\mathcal{G} B | S$ for disjoint $A,B,S \subseteq V$. Then $a \perp_\mathcal{G} b | S^\prime$ for singletons $a\subseteq A, b\subseteq B$ and any $S \subseteq S^\prime \subseteq SAB \backslash ab$. By the previous part of this proof, this implies that $(a,b|S^\prime )^\rceil = (a,b | V\backslash S^\prime ab) \subseteq\mathcal{L}^\rceil \cap [\mathcal{G}]$ for all such $a,b,S^\prime$.  As the family $\{V\backslash S^\prime ab \ : \ S \subseteq S^\prime \subseteq SAB\backslash ab \}$ is equivalent to $\{\tilde{S} \ : \ V\backslash SAB \subseteq \tilde{S} \subseteq V \backslash Sab = (V\backslash SAB) AB \backslash ab\}$, it follows by localizability that $(A,B | V\backslash SAB) = (A,B|S)^\rceil \in \mathcal{L}^\rceil \cap [\mathcal{G}]$, completing the assertion.

$(\Rightarrow):$ Assume now that $\mathcal{L}$ is bidirected global Markov with respect to $\mathcal{G}$. Then by definition, $\mathcal{L}\cap [\mathcal{G}]^\rceil = [\mathcal{G}]^\rceil$. As by \cite[p. 108]{Matus1997}, $[\mathcal{G}]$ is $(L)$ and $(P)$, it remains to note that localizability is preserved under dualization and that the dual of a pseudographoid is a reverse pseudographoid ((Lemma \ref{lem:dual_operator}). 
\end{proof}

Theorem \ref{thm:pseudo_cov} can also be proven using the concept of duality, outlined in Section \ref{sec:duality}. After applying Lemma \ref{lem:dual_operator}, Theorem \ref{thm:pseudo_cov} is seen to follow directly from Theorem \ref{thm:pseudo_conc} after a very short proof.  

\begin{proof}[\textbf{Alternate proof of Theorem \ref{thm:pseudo_cov}}]
Since $\mathcal{L}$ is bidirected pairwise Markov with respect to $\mathcal{G}$, by definition 
\begin{align*}
(a,b)\not\in E \Rightarrow (a,b|\emptyset) \in \mathcal{L} \iff(a,b|V\backslash ab)\in \mathcal{L}^\rceil. 
\end{align*}
Hence, $\mathcal{L}^\rceil$ is undirected pairwise Markov with respect to $\mathcal{G}$. 

Therefore by Theorem \ref{thm:pseudo_conc}, $\mathcal{L}^\rceil$ is undirected global Markov with respect to $\mathcal{G}$ if and only if $\mathcal{L}^\rceil \cap [\mathcal{G}]$ is localizable $(L)$ and closed under $(P)$. By Lemma \ref{lem:dual_operator}, the foregoing statement is true if and only if $(\mathcal{L}^\rceil \cap [\mathcal{G}])^\rceil = \mathcal{L} \cap [\mathcal{G}]^\rceil$ is localizable $(L)$ and closed under $(R)$.

\end{proof}

Thus, duality has allowed for a much shorter and simpler proof Theorem \ref{thm:global_cov} through reuse of Theorem \ref{thm:global_conc}. 

\section{Duality, Semigraphoids, and Pseudographoid Rules}
\label{sec:semi_and_pseudo}

In Section \ref{sec:pseudo_global}, the pairwise and global Markov properties were investigated using the pseudographoid and reverse pseudographoid rules. We now examine pseudographoid and reverse pseudographoid rules in the context of semigraphoids and localizability.

The pseudographoid and reverse pseudographoid rules are  weaker than the intersection and composition rules typically used to study Markov properties when considered on a general relation $\mathcal{L}$. However, undirected and bidirected graphs are defined solely to study systems of random variables, which are semigraphoids by default. Therefore, it makes sense to consider the pseudographoid and reverse pseudographoid rules as they relate to relations of the form $\mathcal{L} = [\mathbf{X}]$ for some random variable $\mathbf{X}$ with finite index set $V$. 

In this section, we demonstrate that when restricted to semigraphoids, closure under the pseudographoid and intersection rules are equivalent. The analogous statement is true for the reverse pseudographoid and composition rule. We demonstrate both equivalences in their own right, but further show how one equivalence can be dualized into the other for a simpler result. 

Furthermore, we show that localizable pseudographoids are a less restrictive class of relations than semigraphoids satisfying intersection. The dual result follows immediately. In particular, this shows that indeed our conditions for equivalences between pairwise and global Markov properties in Section \ref{sec:pseudo_global} are weaker than those which currently exist, for the space of relations.

As remarked in Section \ref{subsec:cond_independence}, for any random vector $\mathbf{X}$, the conditional independence structure $\mathcal{L} = [\mathbf{X}]$ is a semigraphoid, satisfying symmetry $(S)$, decomposition $(D)$, weak union $(U)$, and contraction $(C)$. By Lemma \ref{lem:semi_parsim}, this is equivalent to $\mathcal{L}$ satisfying 
\begin{align*}
(A,BC|S) \in \mathcal{L} \iff (A,B|S), (A,C|SB) \in \mathcal{L}.
\end{align*}
Similar parsimonious equivalences can be stated under the further assumption of intersection, composition, or the pseudographoid rules.

\begin{lemma}
\label{lemma:equivalences}
Let $\mathcal{L}$ be a semigraphoid. Then:
\begin{enumerate}
\item Closure of $\mathcal{L}$ under intersection is equivalent to 
\[
(A,BC|S) \in \mathcal{L} \iff (A,B|SC), (A,C|SB) \in \mathcal{L}.
\] 
\item Closure of $\mathcal{L}$ under composition is equivalent to 
\[
(A,BC|S) \in \mathcal{L} \iff (A,B|S), (A,C|S) \in \mathcal{L}.
\]
\item Closure of $\mathcal{L}$ under the pseudographoid rule is equivalent to 
\[
(a,bc|S) \in \mathcal{L}  \iff (a,b|Sc), (a,c|Sb) \in \mathcal{L}.
\] 
\item Closure of $\mathcal{L}$ under the reverse pseudographoid rule is equivalent to 
\[
(a,bc|S) \in \mathcal{L} \iff (a,b|S), (a,c|S) \in \mathcal{L} .
\]
\end{enumerate}

\end{lemma}

\begin{proof}
The $(\Leftarrow)$ direction of $(1)$ is by definition, while $(\Rightarrow)$ follows from the weak union rule (as defined in Section \ref{subsec:relations}). Similarly, the $(\Leftarrow)$ direction of $(2)$ is by definition while the $(\Rightarrow)$ direction follows by decomposition.

To prove $(3)$, note that if $\mathcal{L}$ is a semigraphoid closed under the pseudographoid rule, then for $a,b,c$ singletons, 
\begin{align*}
(a,b|Sc), (a,c|Sb) \in \mathcal{L} \Rightarrow_{(P)} (a,b|S), (a,c|S) \in \mathcal{L})\\
(a,b|Sc), (a,c|S) \in \mathcal{L} \Rightarrow_{(C)} (a,bc|S) \in \mathcal{L} \\\
(a,bc|S) \in \mathcal{L} \Rightarrow_{(U)} (a,b|Sc), (a,c|Sb) \in \mathcal{L},
\end{align*}
where for clarity we underscore implications with letter of the corresponding rule as given in Section \ref{subsec:cond_independence}. The other direction follows by definition. 

To prove $(4)$, note that
\begin{align*}
(a,b|S), (a,c|S) \in \mathcal{L} \Rightarrow_{(R)} (a,b|Sc), (a,c|Sb) \in \mathcal{L} \\
(a,b|S), (a,c|Sb) \in \mathcal{L} \Rightarrow_{(C)} (a,bc|S) \in \mathcal{L}
\end{align*}. 

\end{proof}

The following lemma shows equivalence between the pseudographoid and intersection rules, when the semigraphoid rules are satisfied. 

\begin{lemma}
\label{lemma:pseudo_inter_equiv}
Let $\mathcal{L} \subseteq \mathcal{T}(V)$ be a semigraphoid. Then $\mathcal{L}$ is closed under the pseudographoid rule if and only if it is closed under intersection.
\end{lemma}
\begin{proof}
The $(\Leftarrow)$ direction is clear; we will show the $(\Rightarrow)$ direction. Define first the generalized pseudographoid rule by
\begin{align*}
(A,B|SC), (A,C|SB) \in \mathcal{L} \Rightarrow (A,BC|S) \in \mathcal{L},\ : \ |A|\leq l, |B| \leq m, |C| \leq n. 
\end{align*}
This rule will be denoted by $(P_{lmn})$; note that the pseudographoid rule on a semigraphoid is $(P_{111})$ while intersection is $(P_{|V|-3, |V|-3, |V|-3})$. Starting then with $(P_{111})$, by induction on $l$, $m$, and $n$. 

\subsubsection*{Claim \ref{lemma:pseudo_inter_equiv}.1}
If the semigraphoid $\mathcal{L}$ is closed under $(P_{k11})$ for some $k$, it is also closed under $(P_{(k+1)11})$. 

To prove Claim \ref{lemma:pseudo_inter_equiv}.1, fix some pairwise disjoint $A,a,b,c,S$ with $|A| = k$ and $a,b,c$ singletons and assume $(Aa,b|Sc), (Aa,c|Sb) \in \mathcal{L}$. Then first,
\begin{align*}
(Aa,b|Sc) \in \mathcal{L} &\Rightarrow_{(D)} (A,b|Sc) \in \mathcal{L}\\
(Aa, c|Sb) \in \mathcal {L} & \Rightarrow_{(D)} (A,c|Sb) \in \mathcal{L}\\
(A,b|Sc), A,c|Sb)  \in \mathcal{L} &\Rightarrow_{(P_{k11})}(A,bc|S) \in \mathcal{L}.  
\end{align*}
Second,
\begin{align*}
(Aa,b | Sc) \in \mathcal{L} & \Rightarrow_{(U)} (a,b|SAc) \in \mathcal{L}\\
(Aa, c | Sb) \in \mathcal{L} & \Rightarrow_{(U)} (a,c|SAb) \in \mathcal{L}\\
(a,b|SAc), (a,c|SAb) \in \mathcal{L} &\Rightarrow_{(P_{k11})} (a,bc|SA) \in \mathcal{L}.
\end{align*}
Finally,
\begin{align*}
(A, bc|S), (a,bc|SA) \in \mathcal{L} \Rightarrow_{(C)} Aa \ci bc | S. 
\end{align*}
which shows closure under $(P_{(k+1)11})$. 

The next claim shows induction on the next coordinate. 
\subsubsection*{Claim \ref{lemma:pseudo_inter_equiv}.2}
If the semigraphoid $\mathcal{L}$ is closed under $(P_{(|V|-3)k1})$ for some $k$, it is also closed under $(P_{(|V|-3)(k+1)1})$. 

Assume that for some pairwise disjoint $A,B,b,c,S$ with $|B| = k$ and $b,c$ singletons, $(A,Bb|Sc), (A,c|SBb) \in \mathcal{L}$. Then 
\begin{align*}
(A,Bb|Sc) \in \mathcal{L} &\Rightarrow_{(U)} (A,B|Sbc) \in \mathcal{L}\\
(A,B|Sbc), (A,c|SBb) \in \mathcal{L} & \Rightarrow_{(P_{(|V|-2)k1})} (A,Bc|Sb) \in \mathcal{L}\\
(A,Bc|Sb) \in \mathcal{L} &\Rightarrow_{(D)} (A,c|Sb) \in \mathcal{L}\\
(A, Bb | Sc) \in \mathcal{L} & \Rightarrow_{(D)} (A, b | Sc) \in \mathcal{L} \\
(A,c|Sb), (A, b | Sc) \in \mathcal{L} &\Rightarrow_{(P_{(|V|-3)k1})} (A,bc|S) \in \mathcal{L}\\
(A,bc|S) (A,B|Sbc) \in \mathcal{L} & \Rightarrow_{(C)}(A,Bbc| S) \in \mathcal{L}. 
\end{align*}
which proves Claim \ref{lemma:pseudo_inter_equiv}.2. 

The next claim completes the inductive argument.
\subsubsection*{Claim \ref{lemma:pseudo_inter_equiv}.3}
If the semigraphoid $\mathcal{L}$ is closed under $(P_{(|V|-3)(|V|-3)k})$ for some $k$, it is also closed under $(P_{(|V|-3)(|V|-3)(k+1})$. 

For disjoint $A,B,C,c,S$ with $|C| = k$ and $|c| = 1$, assume that $(A,B|SCc), (A,Cc|SB) \in \mathcal{L}$. Then
\begin{align*}
(A,Cc|SB) \in \mathcal{L}  & \Rightarrow_{(U)}(A,C|SBc) \in \mathcal{L} \\
(A,B|SCc), (A,C|SBc) \in \mathcal{L} &\Rightarrow_{(P_{(|V|-3)(|V|-3)k})} (A,BC|Sc) \in \mathcal{L}\\
(A,BC|Sc) \in \mathcal{L} &\Rightarrow_{(D)} (A,B|Sc) \in \mathcal{L}\\
(A,Cc|SB) \in \mathcal{L} &\Rightarrow_{(D)} (A,c|SB) \in \mathcal{L} \\
(A,c|SB), (A,B|Sc) \in \mathcal{L} &\Rightarrow_{(P_{(|V|-3)(|V|-3)k})} (A, Bc|S) \in \mathcal{L} \\
(A, Bc|S), (A,C|SBc) \in \mathcal{L} &\Rightarrow_{(C)} (A,BCc|S) \in \mathcal{L} 
\end{align*}
which completes the proof of Claim \ref{lemma:pseudo_inter_equiv}.3. The three claims prove the lemma; the pseudographoid property is equivalent to intersection in view of the semigraphoid properties. 

\end{proof}

The following lemma shows that closure under composition and the reverse pseudographoid rule are equivalent for pseudographoids when the semigraphoid rules are satisfied.

\begin{lemma}
\label{lemma:pseudo_comp_equiv}
Let $\mathcal{L} \subseteq \mathcal{T}(V)$ be a semigraphoid. Then $\mathcal{L}$ is closed under the reverse pseudographoid rule if and only if it is closed under intersection. 
\end{lemma}
\begin{proof}
A direct proof is given in the appendix and should be compared with the direct proof of the dual result in Lemma \ref{lemma:pseudo_inter_equiv}.
\end{proof}

Although we have provided a full proof of Lemma \ref{lemma:pseudo_comp_equiv} in the appendix, the result can be been proven in just a few lines using the dualization procedure of Section \ref{sec:duality}; once more underscoring the power of duality. 

\begin{proof}[\textbf{Alternate proof of Lemma \ref{lemma:pseudo_comp_equiv}}]
Assume first that $\mathcal{L}$ be a semigraphoid closed under the pseudographoid rule. Then by Lemma \ref{lem:dual_operator}, $\mathcal{L}^\rceil$ is a semigraphoid which is closed under the reverse pseudographoid rule. Hence by Lemma \ref{lemma:pseudo_inter_equiv}, $\mathcal{L}^\rceil$ is closed under the intersection rule, which implies after again applying Lemma \ref{lem:dual_operator} that $(\mathcal{L}^\rceil)^\rceil = \mathcal{L}$ is closed under composition. 

The reverse direction is analogous. 
\end{proof}

We have shown that the pseudographoid (resp. reverse pseudographoid) rule is equivalent to intersection (resp. composition) on the set of semigraphoids. However, the following lemma shows that localizability and the pseudographoid (resp. reverse pseudographoid) rule is a weaker set of rules than the semigraphoid rules along with intersection (resp. composition).
\begin{lemma}
For any relation $\mathcal{L}$,
\begin{enumerate}
\item closure of $\mathcal{L}$ under the semigraphoid rules and intersection $(SDUCI)$ implies closure of $\mathcal{L}$ under symmetry, localizability, and the pseudographoid rule $(SLP)$. However, the converse does not hold. 
\item closure of $\mathcal{L}$ under the semigraphoid rules and composition $(SDUCM)$ implies closure of $\mathcal{L}$ under symmetry, localizability, and the reverse pseudographoid rule $(SLP)$. However, the converse does not hold. 
\end{enumerate}
\end{lemma}
\begin{proof}
To prove $1.$, first note that all semigraphoids are localizable as in Lemma \ref{lem:semi_local_comparison}. Furthermore, intersection implies the pseudographoid rule. To see that the converse does not hold, consider the relation $\mathcal{L} = \{ (a, b | \emptyset), (a, b | c)\}$ which localizable and closed under the pseudographoid rule, but is not a semigraphoid. 

The proof of $2.$ is analogous. 
\end{proof}

We have concluded that with respect to semigraphoids (and hence $[\mathbf{X}]$ where $\mathbf{X}$ is a random variable), the pseudographoid (respectively, reverse pseudographoid) and intersection (respectively, composition) properties are equivalent. This was done both directly, and also by using the concept of duality with the latter. However, the semigraphoid rules along with intersection/composition are stronger than localizability along with the pseudographoid/reverse pseudographoid rule. The results of Section \ref{sec:pseudo_global} regarding the pairwise and global Markov properties should be interpreted in this context. 

\section{Duality and Faithfulness}
\label{sec:faithful_trees}

Recall that the global Markov properties generally required containment of $[\mathcal{G}]$ or $[\mathcal{G}]^\rceil$ within $\mathcal{L}$. The notion of \emph{faithfulness} is the reverse inclusion; $\mathcal{L}$ is said to be
\begin{itemize}
\item[a.] \emph{Undirected faithful} to $\mathcal{G}$ if $(A,B |S) \in \mathcal{L} \Rightarrow (A,B|S) \in \mathcal{G}$, that is $\mathcal{L} \subseteq [\mathcal{G}]$, and
\item[b.] \emph{Bidirected faithful} to $\mathcal{G}$ if  $(A,B |S)^\rceil \in \mathcal{L} \Rightarrow (A, B|S) \in \mathcal{G}$, that is $\mathcal{L}^\rceil \subseteq [\mathcal{G}]$. 
\end{itemize}
In some places in the literature faithfulness is defined simple as equivalence between $\mathcal{L}$ (or $\mathcal{L}^\rceil$) and $[\mathcal{G}$. The authors believe the one-way inclusion makes more sense, since then faithfulness of $\mathcal{L}$ to $\mathcal{G}$ says that every triple $\mathcal{L}$ is somehow encoded in $\mathcal{G}$. This allows for $\mathcal{G}$ to be more complex than $\mathcal{L}$. On the other hand, the undirected global Markov property (e.g.) requires that $[\mathcal{G}]$ is contained in $\mathcal{L}$, that is, $\mathcal{G}$ is less complex than $\mathcal{L}$. This interpretation highlights the trade-off between complexity of a graphical model $\mathcal{G}$ when used to model a conditional independence structure $[\mathbf{X}]$. 

\citet{Becker2005} provide conditions under with a relation is undirected faithful to a undirected tree. A weaker result exists for bidirected Markov trees, in the Gaussian setting, due to \citet{Malouche2011}. In Section \ref{sec:faithful_trees}, the results of \citet{Becker2005} for undirected trees will be related those of \citet{Malouche2011} for bidirected trees, and the latter will be strengthened.

In this section, we investigate sufficient conditions under which a relation is faithful to its bidirected graph. We begin with the Gaussian case, and in this regard give an alternate proof of the main theorem of \citet{Malouche2011}. We also note that faithfulness for bidirected graphs, aside from the Gaussian setting, has not been considered by \citet{Malouche2011}. We proceed to show how the result of \citet{Malouche2011} can be easily generalized by dualizing the result of \citet{Becker2005}, and through this technique extend to non-Gaussian settings. In this case, dualization results directly in a stronger result than a parallel result in the literature, demonstrating the power of the technique.

\subsection{The Gaussian Case}
We begin by stating the main result of \citet{Malouche2011} in the notation of relations. 
\begin{thm}[\citet{Malouche2011}]
\label{thm:gaussian_cov_trees}
Let $\mathbf{X} = \{\mathbf{X}_v\}_{v\in V}$ be a Gaussian random vector with bidirected graph $\mathcal{G}_{bi}(\mathbf{X})  = (V, E_{bi}(\mathbf{X}) )$. If $\mathcal{G}_{bi}(\mathbf{X}) $ is a disjoint union of trees, then $\mathbf{X}$ is bidirected faithful to $\mathcal{G}_{un}(\mathbf{X}) $, and hence $[\mathbf{X}] = [\mathcal{G}_{un}(\mathbf{X}) ]$. 
\end{thm}

Note that in Theorem \ref{thm:gaussian_cov_trees} above, the bidirected faithful statement is equivalent to saying that $[\mathbf{X}] \subseteq [\mathcal{G}_{bi}(\mathbf{X}) ]$. The reverse inclusion, $[\mathcal{G}_{bi}(\mathbf{X})]  \subseteq [\mathbf{X}]$, follows from applying the bidirected global Markov property. 

In addition, we will use the following restatement of the main result of \citet{Becker2005} in the Gaussian setting. 
\begin{thm}[\citet{Becker2005}]
\label{thm:gaussian_conc_trees}
Let $\mathbf{X}$ be a Gaussian random vector with undirected graph $\mathcal{G}_{un}(\mathbf{X})$. If $\mathcal{G}_{un}(\mathbf{X}) $ is a disjoint union of trees, then $\mathbf{X}$ is undirected faithful to $\mathcal{G}_{un}(\mathbf{X})$, i.e., $[\mathbf{X}] \subseteq [\mathcal{G}_{un}(\mathbf{X})]$.  
\end{thm}

Lemma \ref{lem:gaussian_duality} allows ``dualization" of Theorem \ref{thm:gaussian_conc_trees} into Theorem \ref{thm:gaussian_cov_trees}, resulting in a much simpler proof below of the result on bidirected faithfulness of Gaussian bidirected Markov trees by \citet{Malouche2011}.

\begin{proof}[\textbf{Alternate proof of Theorem \ref{thm:gaussian_cov_trees}}]
Without loss of generality, assume that $\mathbf{X}$ is distributed as $\mathcal{N}(0,\Sigma)$. Let $\mathbf{Y}$ be a Gaussian random vector distributed as $\mathcal{N}(0, \Sigma^{-1})$. 

From Lemma \ref{lem:gaussian_duality}, $[\mathbf{Y}] = [\mathbf{X}]^\rceil$, and hence $\mathbf{X}_a \ci \mathbf{X}_b \iff \mathbf{Y}_a \ci \mathbf{Y}_b | \mathbf{Y}_{V\backslash ab}$. Therefore, by construction, $\mathcal{G}_{bi}(\mathbf{X}) = \mathcal{G}_{un}(\mathbf{Y})$. By assumption, $\mathcal{G}_{bi}(\mathbf{X})$ and hence also $\mathcal{G}_{un}(\mathbf{Y})$ is a disjoint union of trees, and it follows from Theorem \ref{thm:gaussian_conc_trees} that $\mathbf{Y}$ is undirected faithful to $\mathcal{G}_{un}(\mathbf{Y})$, i.e., $[\mathbf{Y}] \subseteq [\mathcal{G}_{un}(\mathbf{Y})]$.

As $\mathbf{Y}$ is Gaussian, $[\mathbf{Y}]$ satisfies the intersection rule, and therefore $[\mathbf{Y}]$ is undirected global Markov \citep{Pearl1986} with respect to its undirected graph $[\mathcal{G}_{un}(\mathbf{Y})]$, which in turn gives $[\mathcal{G}_{un}(\mathbf{Y})] \subseteq [\mathbf{Y}]$. Therefore, $[\mathcal{G}_{un}(\mathbf{Y})] = [\mathbf{Y}]$, and hence $[\mathcal{G}_{bi}(\mathbf{X})] = [\mathcal{G}_{un}(\mathbf{Y})] = [\mathbf{Y}]$.

Since $[\mathbf{Y}] = [\mathbf{X}]^\rceil$ from Lemma \ref{lem:gaussian_duality}, it follows that $[\mathcal{G}_{bi}(\mathbf{X})] = [\mathbf{X}]^\rceil$, and hence by definition $\mathbf{X}$ is bidirected faithful to $\mathcal{G}_{bi}(\mathbf{X})$.
\end{proof}

\subsection{The General Case}
Using the same dualization technique as used to prove Theorem \ref{thm:gaussian_cov_trees}, it is possible to produce a more general result concerning bidirected faithful trees directly from the main result of \citet{Becker2005}. 

We now state the general result given by \citet{Becker2005} regarding undirected faithfulness of random variables to tree graphs. 
\begin{thm}[\citet{Becker2005}]
\label{thm:conc_trees}
Let $\mathbf{X}$ be a random vector with $[\mathbf{X}]$ closed under intersection and decomposable transitivity.  If $\mathcal{G}_{un}(\mathbf{X})$ is a disjoint union of trees, then $[\mathbf{X}] = [\mathcal{G}_{un}(\mathbf{X})]$, i.e., $\mathbf{X}$ is both undirected global Markov and undirected faithful with respect to $\mathcal{G}_{un}(\mathbf{X})$.
\end{thm}

In the statement of Theorem \ref{thm:conc_trees}, undirected faithfulness is equivalent to stating that $[\mathbf{X}] \subseteq [\mathcal{G}_{un}(\mathbf{X})]$. The other inclusion follows from the fact that $[\mathbf{X}]$ is undirected global Markov with respect to $\mathcal{G}_{un}(\mathbf{X})$. We now proceed to prove, in Theorem \ref{thm:cov_trees} below, that all relations are bidirected faithful to their bidirected trees under the composition and dual decomposable transitivity assumptions. Theorem \ref{thm:cov_trees} ``dualizes" Theorem \ref{thm:conc_trees}, using the same method as was used for the Gaussian case. 
\begin{thm}
\label{thm:cov_trees}
Let $\mathbf{X}$ be a random vector with $[\mathbf{X}]$ closed under composition and dual decomposable transitivity. If $\mathcal{G}_{bi}(\mathbf{X})$ is a disjoint union of trees, then $[\mathbf{X}]^\rceil = [\mathcal{G}_{bi}(\mathbf{X})]$, i.e., $\mathbf{X}$ is both bidirected global Markov and bidirected faithful with respect to $\mathcal{G}_{bi}(\mathbf{X})$.
\end{thm}
\begin{proof}
If $[\mathbf{X}]$ is closed under composition and dual decomposable transitivity, then by Lemma \ref{lem:dual_operator}, $[\mathbf{X}]^\rceil$ is closed under decomposable transitivity and intersection. Furthermore, since $(a,b|\emptyset) \in [\mathbf{X}] \iff (a,b|V\backslash ab) \in [\mathbf{X}]^\rceil$, it follows that $\mathcal{G}_{bi}(\mathbf{X})$ is the undirected graph of $[\mathbf{X}]^\rceil$, i.e., $\mathcal{G}_{bi}(\mathbf{X}) = \mathcal{G}_{un}(\mathbf{X}^\rceil)$. As $\mathcal{G}_{bi}(\mathbf{X}) = \mathcal{G}_{un}(\mathbf{X}^\rceil)$ is assumed to be a disjoint union of trees, it follows by Theorem \ref{thm:conc_trees} that $[\mathbf{X}]^\rceil = [\mathcal{G}_{bi}(\mathbf{X})]$, completing the proof.  
\end{proof}
In Theorem \ref{thm:cov_trees}, bidirected faithfulness is the statement that $[\mathbf{X}]^\rceil \subseteq [\mathcal{G}_{bi}(\mathbf{X})]$. The reverse inclusion follows from the bidirected global Markov property. 

We now recover the exact theorem of \citet{Malouche2011} as a special case of Theorem \ref{thm:cov_trees}. It remains only to show that the Gaussian distribution satisfies dual decomposable transitivity. This property can also be shown using duality, as in the following lemma. 
\begin{lemma}
\label{lemma:gaussian_dual_decomp}
Let $\mathbf{X}$ be a Gaussian random vector. Then $[\mathbf{X}]$ is closed under dual decomposable transitivity. 
\end{lemma}
\begin{proof}
Without loss of generality, assume $\mathbf{X}$ is distributed as $\mathcal{N}(0,\Sigma)$, and let $\mathbf{Y}$ be a $\mathcal{N}(0, \Sigma^{-1})$ random vector. As Gaussian random vectors satisfy decomposable transitivity \citep{Becker2005}, $[\mathbf{Y}]$ is closed under decomposable transitivity.  By Lemma \ref{lem:gaussian_duality}, $[\mathbf{Y}]^\rceil = [\mathbf{X}]$, and hence $[\mathbf{X}]$ is closed under dual decomposable transitivity. 
\end{proof}

In view of Lemma \ref{lemma:gaussian_dual_decomp}, Theorem \ref{thm:cov_trees} above indeed has Theorem \ref{thm:gaussian_cov_trees} as a special case.

\section{Conclusions}
This paper examined duality and its implications on undirected and bidirected graphical models. The dualization technique described here, which explicitly relates undirected graphs to bidirected graphs in a dual sense, was shown to facilitate many proofs regarding bidirected graphs. In demonstrating this method, new results regarding such graphs were derived.The pseudographoid and reverse pseudographoid rules were shown (along with localizability) to be sufficient conditions for equivalence between pairwise and global Markov properties. The pseudographoid and reverse pseudographoid rules were also shown to be equivalent to the intersection and composition rules in view of the semigraphoid axioms.  Finally, we demonstrated that the bidirected faithful tree proof of \cite{Malouche2011} can be accomplished by combining the previous result for undirected faithfulness with the notion of duality, yielding a more general result than existed previously. 

\section*{Acknowledgements}
The authors would like to thank Kayvon Sadeghi and Milan Studen\'{y} for useful discussions. 

\bibliographystyle{plainnat}
\bibliography{refs_duality}

\begin{thebibliography}{18}
\providecommand{\natexlab}[1]{#1}
\providecommand{\url}[1]{\texttt{#1}}
\expandafter\ifx\csname urlstyle\endcsname\relax
  \providecommand{\doi}[1]{doi: #1}\else
  \providecommand{\doi}{doi: \begingroup \urlstyle{rm}\Url}\fi

\bibitem[Banerjee and Richardson(2003)]{Banerjee2003}
M.~Banerjee and T.~Richardson.
\newblock On a dualization of graphical gaussian models: A correction note.
\newblock \emph{Scandinavian Journal of Statistics}, 30\penalty0 (4):\penalty0
  817--820, 2003.

\bibitem[Becker et~al.(2005)Becker, Geiger, and Meek]{Becker2005}
A.~Becker, D.~Geiger, and C.~Meek.
\newblock Perfect tree-like markovian distributions.
\newblock \emph{Probability and Mathematical Statistics}, 25\penalty0
  (2):\penalty0 231--239, 2005.

\bibitem[Cox and Wermuth(1993)]{Cox1993}
D.~R. Cox and Nanny Wermuth.
\newblock Linear dependencies represented by chain graphs.
\newblock \emph{Statistical Science}, 8\penalty0 (3):\penalty0 204--218, 1993.

\bibitem[Cox and Wermuth(1996)]{Cox1996}
D.R. Cox and N.~Wermuth.
\newblock \emph{Multivariate Dependencies: Models, Analysis, and
  Interpretation}.
\newblock Chapman and Hall, 1996.

\bibitem[Kauermann(1996)]{Kauermann1996}
G.~Kauermann.
\newblock On a dualization of graphical gaussian models.
\newblock \emph{Scandinavian Journal of Statistics}, 23\penalty0 (1):\penalty0
  105--116, 1996.

\bibitem[Lauritzen(1996)]{Lauritzen1996}
S.L. Lauritzen.
\newblock \emph{Graphical models}.
\newblock Oxford University Press, 1996.

\bibitem[Lnenicka and Mat\'{u}s(2007)]{Lnenicka2007}
R.~Lnenicka and F.~Mat\'{u}s.
\newblock On gaussian conditional independence structures.
\newblock \emph{Kybernetika}, 43\penalty0 (3):\penalty0 327--342, 2007.

\bibitem[Malouche and Rajaratnam(2011)]{Malouche2011}
D.~Malouche and B.~Rajaratnam.
\newblock Gaussian covariance faithful markov trees.
\newblock \emph{Journal of Probability and Statistics}, 2011.

\bibitem[Mat\'{u}s(1992{\natexlab{a}})]{Matus1992a}
F.~Mat\'{u}s.
\newblock Ascending and descending conditional independence relations.
\newblock In \emph{Information Theory, Statistical Decision Functions, and
  Random Processes}, Transactions of the Eleventh Prague Conference B, pages
  189--200, 1992{\natexlab{a}}.

\bibitem[Mat\'{u}s(1992{\natexlab{b}})]{Matus1992b}
F.~Mat\'{u}s.
\newblock On equivalence of markov properties over undirected graphs.
\newblock \emph{Journal of Applied Probability}, 29\penalty0 (3):\penalty0
  745--749, 1992{\natexlab{b}}.

\bibitem[Mat\'{u}s(1997)]{Matus1997}
F.~Mat\'{u}s.
\newblock Conditional independence structures examined via minors.
\newblock \emph{Annals of Mathematics and Artificial Intelligence},
  21:\penalty0 99--128, 1997.

\bibitem[Pearl and Paz(1986)]{Pearl1986}
P.~Pearl and A.~Paz.
\newblock Graphoids: Graph-based logic for reasoning about relevance relations
  or when would x tell you more about y if you already know z?
\newblock In \emph{ECAI}, pages 357--363, 1986.

\bibitem[Pe{\~n}a(2013)]{Pena2013}
Jose~M. Pe{\~n}a.
\newblock Reading dependencies from covariance graphs.
\newblock \emph{International Journal of Approximate Reasoning}, 54\penalty0
  (1):\penalty0 216--227, 2013.

\bibitem[Studen\'{y}(1992)]{Studeny1992}
M.~Studen\'{y}.
\newblock Conditional independence relations have no finite complete
  characterization.
\newblock In \emph{Information Theory, Statistical Decision Functions, and
  Random Processes}, Transactions of the Eleventh Prague Conference B, pages
  377--396, 1992.

\bibitem[Studen\'{y}(1997)]{Studeny1997}
M.~Studen\'{y}.
\newblock Semigraphoids and structures of probabilistic conditional
  independence.
\newblock \emph{Annals of Mathematics and Artificial Intelligence}, 21\penalty0
  (1):\penalty0 71--98, January 1997.

\bibitem[Studen\'{y}(2001)]{Studeny2001}
M.~Studen\'{y}.
\newblock On mathematical description of probabilistic conditional independence
  structures.
\newblock Technical report, Academy of Science of the Czech Republic, 2001.

\bibitem[Vomlel and Studen\'{y}(2007)]{Vomlel2007}
J.~Vomlel and M.~Studen\'{y}.
\newblock Graphical and algebraic representatives of conditional independence
  models.
\newblock In \emph{Advances in Probabilistic Graphical Models}. Springer, 2007.

\bibitem[Whittaker(1990)]{WhittakerBook}
J.~Whittaker.
\newblock \emph{{Graphical Models in Applied Multivariate Statistics}}.
\newblock Wiley, 1990.

\end{thebibliography}

\newpage
\appendix
\section{Additional Proofs}
\begin{lemma}
\label{lemma:pseudo_conc_necessary}
For any undirected graph $\mathcal{G} = (V,E)$, $[\mathcal{G}]$ is localizable and closed under the pseudographoid rule $(P)$. 
\end{lemma}

\begin{proof}
We first show closure under $(P)$. Assume that $(a, b|Sc), (a,c|Sb) \in [\mathcal{G}]$. Then by definition $a \perp_\mathcal{G} b | Sc$ and $a \perp_\mathcal{G} c | Sb$. Assume that $a \not\perp_\mathcal{G} b | S$. Then there exists a path $p_{ab}$ connecting $a$ and $b$ which does not intersect $S$. Since $a \perp_\mathcal{G} b | Sc$, $p_{ab}$ must intersect $c$. In this case, $p_{ab}$ can be split into paths $p_{ac}$ and $p_{cb}$, neither of which intersect $S$. But then $p_{ac}$ is a path connecting $a$ and $c$, not intersecting $Sb$. This yields a contradiction, as $(a \perp_\mathcal{G} c | Sb$. Hence, $a \perp_\mathcal{G} b | S$ and similarly $a \perp_\mathcal{G} b | S$, which implies by the undirected global Markov rule that $(a,b|S), (a,c|S) \in \mathcal{L}_\mathcal{G}$. Therefore $[\mathcal{G}]$ is closed under $(P)$. 

Next, we show localizability of $[\mathcal{G}]$. If $(A,B|S) \in [\mathcal{G}]$, by construction $A \perp_\mathcal{G} B | S$. In this case, it follows that $a\perp_\mathcal{G} b | S$ for any singletons $a\subseteq A, b\subseteq B$. This implies that $a\perp_\mathcal{G} b | S^\prime$ for any $S \subseteq S^\prime\subseteq SAB\backslash ab$. It follows by construction of $[\mathcal{G}]$ that $(a,b|S^\prime) \in [\mathcal{G}]$ for any such $S^\prime$, showing the $(\Rightarrow)$ direction of the localizability rule.

Finally, let $A,B,S \subseteq V$ be pairwise disjoint. Assume that $(a,b|S^\prime) \in [\mathcal{G}]$ whenever $a\subseteq A, b\subseteq B$ are singletons with $S^\prime$ satisfying $S \subseteq S^\prime \subseteq SAB\backslash\{a,b\}$. Then by construction of $[\mathcal{G}]$, $a \perp_\mathcal{G} b | S^\prime$ for any choice of $a,b,S^\prime$. In particular, for any choice of $a,b$, $a\perp_\mathcal{G} b | S$, which implies that $A \perp_\mathcal{G} B | S$. It follows that $(A,B|S) \in [\mathcal{G}]$, showing the $(\Leftarrow)$ direction and completing the proof. 
\end{proof}

\begin{proof}[\textbf{Proof of Theorem \ref{thm:pseudo_conc}} ($\Leftarrow$).]
First assume that $\mathcal{L} \cap [\mathcal{G}]$ is $(L)$ and $(P)$. We first show that $a \perp_\mathcal{G} b| S \Rightarrow (a,b|S) \in \mathcal{L} \cap [\mathcal{G}]$ when $a$ and $b$ are singletons by backward induction on $S$. This follows the proof of Theorem \ref{thm:global_conc} by \cite{Pearl1986}. 

If $a \perp_\mathcal{G} b | S$ for $|S| = |V| - 2$, with $|a| = |b| = 1$ and $S = V \backslash ab$, then $(a,b|S) \in \mathcal{L} \cap [\mathcal{G}]$ by the undirected pairwise Markov assumption. Assume now that whenever $a \perp_\mathcal{G} b | S^\prime $ for some $|a| = |b| = 1$ and $|S^\prime| = k\leq |V| -2$, $(a,b|S^\prime) \in \mathcal{L} \cap [\mathcal{G}]$. Let then $a \perp_\mathcal{G} b | S$ with $|S| = k  - 1$. As $|a| = |b| = 1$ and $|S| < |V| -2$, there exists some singleton $c \subseteq V$ disjoint with $Sab$.  If  $a \perp_\mathcal{G} b | S$, then $a \perp_\mathcal{G} b | Sc$, and by the inductive hypothesis it follows that $(a,b|Sc) \in \mathcal{L} \cap [\mathcal{G}]$. 

Next,  $a \perp_\mathcal{G} b | S \land a \perp_\mathcal{G} b | Sc$ implies that $a\perp_\mathcal{G} c |S \lor w \perp_\mathcal{G} b | S$. To see this, note that if $a\not\perp_\mathcal{G} c |S \land c \not\perp_\mathcal{G} b | S$, there exists a path $p_{ac}$ which does not intersect $S$, and a path $p_{cb}$ which does not intersect $S$, Then the path created by joining $p_{ac}$ and $p_{cb}$ is a path from $a$ to $b$ which does not intersect $S$, yielding a contradiction. Assume then, without loss of generality, that $a \perp_\mathcal{G} c | S$. Then $a \perp_\mathcal{G} c | Sb$, and by the inductive hypothesis $(a,c|Sb) \in \mathcal{L}$. As $\mathcal{L}$ satisfies the pseudographoid rule $(P)$, $(a,b|Sc), (a,c|Sb) \in \mathcal{L}  \cap [\mathcal{G}] \Rightarrow (a,b|S) \in \mathcal{L} \cap [\mathcal{G}]$. 

Finally, assume that $A \perp_\mathcal{G} B | S$ for arbitrary disjoint $A,B,S\subseteq V$. Then $a \perp_\mathcal{G} b | S$ for any singletons $a \subseteq A, b\subseteq B$, which implies that $a \perp_\mathcal{G} b | S^\prime$ for any $S\subseteq S^\prime\subseteq SAB\backslash ab$. By the previous argument, $(a,b|S^\prime) \in \mathcal{L} \cap [\mathcal{G}]$ for any such $S^\prime$, which by localizability of $\mathcal{L} \cap [\mathcal{G}]$ implies that $(A,B|S) \in \mathcal{L} \cap [\mathcal{G}]$ as required. It follows that $\mathcal{L} \cap [\mathcal{G}]$ and hence $\mathcal{L}$ is undirected global Markov with respect to $\mathcal{G}$. 
\end{proof}

\begin{proof}[\textbf{Proof of Lemma \ref{lemma:pseudo_comp_equiv}}]
To begin, we define a generalized reverse pseudographoid rule by
\begin{align*}
(A,B|S), (A,C|S) \in \mathcal{L} \Rightarrow (A,B|SC), (A,C|SB) \in \mathcal{L}, \ : \ |A|\leq l, |B| \leq m, |C| \leq n. 
\end{align*}

This rule will be denoted by $(R_{lmn})$. Note that $(R_{111})$ is the reverse pseudographoid rule for a semigraphoid, while $(R_{|V|-3,|V|-3,|V|-3})$ is equivalent to composition. We prove the lemma by induction separately on each of $l$, $m$, and $n$, beginning with $l$ in the following claim.

\subsubsection*{Claim \ref{lemma:pseudo_comp_equiv}.1}
If $\mathcal{L}$ is closed under $(R_{k11})$ for some $k$, it is also closed under $(R_{(k+1)11}$. 

To prove Claim \ref{lemma:pseudo_comp_equiv}.1, assume that for some pairwise disjoint $A,a,b,c,S$ with $|A| = k$ and $a,b,c$ singletons that $(Aa,b|S), (Aa,c|S) \in \mathcal{L}$. Then first,
\begin{align*}
(Aa,b|S) \in \mathcal{L} &\Rightarrow_{(U)} (A,b|Sa) \in \mathcal{L} \\
(Aa,c|S) \in \mathcal{L} &\Rightarrow_{(U)} (A,c|Sa) \in \mathcal{L} \\
(A,b|Sa), (A,c|Sa) \in \mathcal{L} & \Rightarrow_{(R_{k11})} (A,bc|Sa) \in \mathcal{L}. 
\end{align*}
Next, 
\begin{align*}
(Aa,b|S) \in \mathcal{L} & \Rightarrow_{(D)} (a,b|S) \in \mathcal{L}\\
(Aa,c|S) \in \mathcal{L} & \Rightarrow_{(D)} (a,c|S) \in \mathcal{L}\\
a,b|S), (a,c|S) \in \mathcal{L} &\Rightarrow_{(R_{111})} (a,bc|S) \in \mathcal{L}. 
\end{align*}
Finally,
\begin{align*}
(a,bc|S), (A,bc|Sa) \in \mathcal{L} \Rightarrow_{(C)} (Aa,bc|S) \in \mathcal{L}
\end{align*}
which completes the proof of Claim \ref{lemma:pseudo_comp_equiv}.1. 

Claim \ref{lemma:pseudo_comp_equiv}.2 below completes the induction in the second coordinate. 
\subsubsection*{Claim \ref{lemma:pseudo_comp_equiv}.2}
If $\mathcal{L}$ is closed under $(R_{(|V|-3)k1})$ for some $k$, it is also closed under $(R_{(|V|-3)(k+1)1})$. 

To prove the claim, assume that for pairwise disjoint $A,B,b,C,S$ with $|B| = k$ and $|b| = 1$ that $(A,Bb|S), (A,c|S) \in \mathcal{L}$. Then first,
\begin{align*}
(A,Bb|S) \in \mathcal{L} & \Rightarrow_{(D)} (A,B|S) \in \mathcal{L}\\
(A,B|S), (A,c|S) \in \mathcal{L} & \Rightarrow_{(R_{(|V|-3)k1})} (A,Bc|S) \in \mathcal{L}. 
\end{align*}
Next, 
\begin{align*}
(A,Bc|S) \in \mathcal{L} &\Rightarrow_{(U)}  (A,c|SB) \in \mathcal{L} \\
 A,Bb|S) \in \mathcal{L} & \Rightarrow_{(U)}  (A,b|SB) \in \mathcal{L} \\
(A,b|SB), (A,c|SB) \in \mathcal{L}  &\Rightarrow_{(R_{(|V|-3)k1})}  (A,bc|SB) \in \mathcal{L}. 
\end{align*}
Finally, 
\begin{align*}
(A,B|S), (A,bc|SB) \in \mathcal{L} \Rightarrow_{(C)}(A,Bbc|S)
\end{align*}
which proves the claim.

Finally, Claim \ref{lemma:pseudo_comp_equiv}.3 completes the proof.
\subsubsection*{Claim \ref{lemma:pseudo_comp_equiv}.3}
If $\mathcal{L}$ is closed under $(R_{(|V|-3)(|V|-3)k})$ for some $k$, it is also closed under $(R_{(|V|-3)(|V|-3)(k+1)})$. 

To see this, let $A,B,C,c,S$ be pairwise disjoint with $|C| = k$ and $|c| = 1$, and assume that $(A,B|S), (A,Cc|S) \in \mathcal{L}$. Then
\begin{align*}
(A,Cc|S) \in \mathcal{L} &\Rightarrow_{(D)} (A,C|S) \in \mathcal{L} \\
(A,B|S), (A,C|S) \in \mathcal{L}& \Rightarrow_{(R_{(|V|-3)(|V|-3)k})} (A,BC|S) \in \mathcal{L} \\
(A,Cc|S) \in \mathcal{L} & \Rightarrow_{(D)} (A,c|S) \in \mathcal{L}\\
(A,BC|S), (A,c|S) \in \mathcal{L} &\Rightarrow_{(R_{(|V|-3)(|V|-3)k})} (A,BCc|S) \in \mathcal{L}, 
\end{align*}
which proves the claim and completes the proof of the lemma. 
\end{proof}

\end{document}